\documentclass{amsart}
\usepackage{wrapfig}
\usepackage[dvips]{graphicx}
\usepackage{amsmath,amsthm,amssymb,amscd}
\usepackage{braket}
\theoremstyle{definition}
\newtheorem{thm}{Theorem}[section]

\newtheorem{pro}[thm]{Proposition}
\newtheorem{cor}[thm]{Corollary}
\newtheorem{lem}[thm]{Lemma}
\newtheorem{ex}[thm]{Example}

\theoremstyle{definition}
\def\c#1{\mathcal {#1}}
\def\m#1{\mathbb {#1}}
\def\r#1{\rm {#1}}
\begin{document}

\title{Fundamental group of $AF$-algebras with finite dimensional trace space}
\author{Takashi Kawahara}
\email{t-kawahara@math.kyushu-u.ac.jp}      
\maketitle
\begin{abstract}
We consider the realization of fundamental groups of $AF$-algebras in a certain class.  
We find the fundametal groups of $AF$-algebras with finite dimensional trace space which is not realizable as a fundamental group of von Neumann algebras.  
\end{abstract}

\section{Introduction}
Let $\c{M}$ be a factor of type $\r{II}_1$ with a normalized trace 
$\tau$. Murray and von Neumann introduced 
the fundamental group $F(\c{M})$ of $\c{M}$ in \cite{MN}. 
They showed that if $\c{M}$ is  hyperfinite, then 
$F(\c{M}) = {\mathbb R_+^{\times}}$. Since then 
there has been many works on the computation of the 
fundamental groups. Voiculescu showed that 
$F(L(\mathbb{F}_{\infty}))$ of the group factor $L(\mathbb{F}_{\infty})$ 
of the free group $\mathbb{F}_{\infty}$ contains the positive rationals in \cite{Vo} and 
Radulescu proved that 
$F(L(\mathbb{F}_{\infty})) = {\mathbb R}_+^{\times}$ in 
\cite{RF}.  Connes \cite{Co} showed that $F(L(G))$ is a countable group if $G$ is an ICC group with property (T). Popa and Vaes 
showed that either countable subgroup of $\mathbb R_+^{\times}$ or any uncountable group belonging to a certain "large" class 
can be realized as the fundamental group of some 
factor of type $\r{ II}_1$ in \cite{Po} and in \cite{PoVa}.  \\
\ \ \ Nawata and Watatani \cite{NY}, \cite{NY2} introduced the fundamental group of simple $C^*$-algebras 
with unique trace whose study is essentially based on the computation  of 
Picard groups by Kodaka \cite{kod1}, \cite{kod2}, \cite{kod3}.  
Nawata defined the fundamental group of non-unital $C^{*}$-algebras \cite{N1} and calculate the Picard group of some projectionless $C^{*}$-algebras with strict comparison by the fundamental groups \cite{N2}.  \\
\ \ \ We introduced the fundamental group of $C^*$-algebras with finite dimensional trace space in \cite{TK1} 
whose study is essentially based on the computation of \cite{NY}, \cite{NY2}.  
\ \ \ Moreover, We introduced the fundamental group of finite von Neumann algebras with finite dimensional normal trace space in \cite{TK2} and identified the realizable fundamental groups of them.  \\
\ \ \ In this paper, we will consider the difference between fundamental groups of $C^{*}$-algebras and those of finite von Neumann algebras.  
We will consider the realization of fundamental groups of $AF$-algebras.  
Let $a,b$ be non-zero positive real numbers.  
We suppose that $a\neq b$.  
Then we can notice that there is no finite von Neumann algebra $\c{M}$ with finite dimensional trace space 
such that $F(\c{M})=\set{\left[ 
 \begin{array}{cc}
a^{n} & 0  \\
 0 & b^{n} \\
 \end{array} 
 \right]:n\in \m{Z}}$ by the form of fundamental group of von Neumann algebras.  
If $a$ is algebraic number and $b$ is rational (or specific algebraic number), 
then there is no simple $AF$-algebra $\c{A}$ in a certain class 
such that $F(\c{A})=\set{\left[ 
 \begin{array}{cc}
a^{n} & 0  \\
 0 & b^{n} \\
 \end{array} 
 \right]:n\in \m{Z}}$.  
However, we show this as one of the main theorems at last, if $a$ is transcendal, 
then there is a simple $AF$-algebra $\c{A}$  
such that $F(\c{A})=\set{\left[ 
 \begin{array}{cc}
a^{n} & 0  \\
 0 & b^{n} \\
 \end{array} 
 \right]:n\in \m{Z}}$.  
 
 \section{Realization of fundamental group of AF-algebra}

We are interested in the existence of fundamental group.  
We will focus on the fundamental group of the specific $AF$-algebras.  
Let $\c{A}$ be an $AF$-algebra with $n$-dimensional trace space.  
Say $\set{\varphi_{i}}^{n}_{i=1}=\partial_{e} T(\c{A})_{1}$.  
Let $g$ be an element of $K_{0}(\c{A})$.  
Put $\Gamma(a)=(\varphi_{1}^{*}(g),\varphi^{*}_{2}(g),\dots,\varphi^{*}_{n}(g))$.  
From now, we will consider $AF$-algebras which satisfies that $\Gamma$ is an injective unit preserving ordered abellian group homomorphism from $(K_{0}(\c{A}),K_{0}(\c{A})^{+},[1_{\c{A}}])$ to $(\m{R}^{n},\m{R}^{n}_{+},(1,1,\cdots,1))$.  
We call that $\c{A}$ has $K_{0}(\c{A})$ $embedding$ $on$ $\m{R}^{n}$ if $\Gamma$ is an injective unit preserving ordered abellian group homomorphism, 
Then we can see that $(K_{0}(\c{A}),K_{0}(\c{A})^{+},[1_{\c{A}}])$ is an ordered abellian subgroup of $(\m{R}^{n},\m{R}^{n}_{+},(1,1,\cdots,1))$.   
Conversely, we will conduct the condition that an ordered abellian subgroup of $(\m{R}^{n},\m{R}^{n}_{+},(1,1,\cdots,1))$ is considered as $(K_{0}(\c{A}),K_{0}(\c{A})^{+},[1_{\c{A}}])$ for some $\c{A}$ the dimension of trace space of which is $n$.  

\begin{pro}\label{pro:embedding}
Let $\c{A}$ be an $AF$-algebra with $n$-dimensional trace space.  
We suppose that  $K_{0}(\c{A})$ is embedding on $\m{R}^{n}$.  
Let $D$ be an invertible positive diagonal matrix of $M_{n}(\m{C})$ and let $U$ be a permutation unitary of $M_{n}(\m{C})$.   
Then $\Gamma(K_{0}(\c{A}))DU=\Gamma(K_{0}(\c{A}))$ if and only if $DU\in F(\c{A})$.   
\end{pro}

\begin{proof}
We consider that $K_{0}(\c{A})\subset \m{R}^{n}$ and that $\Gamma(K_{0}(\c{A}))=K_{0}(\c{A})$.  
We suppose that $K_{0}(\c{A})DU=K_{0}(\c{A})$.  
Let $p$ be a corresponding projection of $M_{k}(\c{A})$ to $(1,1,\cdots,1)DU$.  
We define an isomorphism of ordered abelian group 
from $(K_{0}(\c{A}),K_{0}(\c{A})^{+},(1,1,\cdots,1)$ onto $(K_{0}(\c{A}),K_{0}(\c{A})^{+},(1,1,\cdots,1)DU)$ 
$M_{DU}$ by \\ 
$M_{DU}g=g(DU)$.  Then an isomorphism $\Phi \c{A}\rightarrow pM_{k}(\c{A})p$ is induced   
by $M_{DU}$ and the matrix representation of $T_{(k,p,\Phi)}$ is $DU$.  
Therefore $DU\in F(\c{A})$.  
Converse is followed in \cite{TK1}.  
\end{proof}

\begin{pro}\label{pro:dense1}
Let $M$ be an additive subgroup of $\m{R}^{n}$.  
Let $\varepsilon>0$ arbitrary.  
We suppose that there exist $\set{a^{\varepsilon}_{i}}^{n}_{i=1}\subset M$ satisfying following properties;\\
(i) $\set{a^{\varepsilon}_{i}}^{n}_{i=1}$ are $\m{R}$-linearly independent. \\ 
(ii)$||a^{\varepsilon}_{i}||< \varepsilon$, 
where $||\cdot||$ is an Euclidean norm of $\m{R}^{n}$.\\
Then $M$ is dense in $\m{R}^{n}$.  
\end{pro}

\begin{proof}
Let $M$ be an additive subgroup of $\m{R}^{n}$ satisfying the hypothesis 
and let $\varepsilon>0$.   
Then we can choose $\set{a^{\frac{\varepsilon}{n}}_{i}}^{n}_{i=1}$.
By (i), $\set{a^{\frac{\varepsilon}{n}}_{i}}^{n}_{i=1}$ is a basis of $\m{R}^{n}$.  
Let $a$ be in $\m{R}^{n}$.  
Then $a=\sum^{n}_{i=1}r_{i}a^{\frac{\varepsilon}{n}}_{i}$ for some $r_{i}$ in $\m{R}$.  
Let $m_{i}$ be an integer satisfying $m_{i} \leq r_{i} \leq m_{i}+1$.  
Then $||a-\sum^{n}_{i=1}m_{i}a^{\frac{\varepsilon}{n}}_{i}||\leq \sum^{n}_{i=1}(r_{i}-m_{i})||a^{\frac{\varepsilon}{n}}_{i}||< \varepsilon$ by (ii).  
Therefore $M$ is dense in $\m{R}$.  
\end{proof}

Let $M$ be an additive subgroup of $\m{R}^{n}$.  
We define the positive cone of $M$ by $\set{(h_{i})^{n}_{i=1}\in M:h_{i}>0}\cup\set{(0,0,\cdots,0)}$.  
We denote it by $M^{+}$.  
Then $(M,M^{+})$ is an ordered abellian group.  
Moreover, we define the additive group homomorphism $\varphi_{k}$ from $M$ into $\m{R}$ 
by $\varphi_{k}((h_{i})^{n}_{i=1}))=h_{k}$.  
Then $\varphi_{k}$ are positive homomorphisms.  
Let $\set{g_{n}}^{\infty}_{n=1}$ be a sequence of $M$.  
We call $\set{g_{n}}^{\infty}_{n=1}$ $increasing$ ($decreasing$) 
if $g_{n}\leq g_{m}$ ($g_{n}\geq  g_{m}$) for any $n<m$ .  
We denote by $S(M,M^{+},(1,1,\cdots,1))$ the state space of $(M,M^{+},(1,1,\cdots,1))$ 
and by $\partial_{e}S(M,M^{+},(1,1,\cdots,1))$ the set of extremal points of $S(M,M^{+},(1,1,\cdots,1))$.  

\begin{pro}\label{pro:dense2}
Let $(M,M^{+})$ be an ordered abellian subgroup of $(\m{R}^{n},(\m{R}^{n})^{+})$.  
We suppose that $M$ is dense in $\m{R}^{n}$ and $(1,1,\cdots,1)\in M$.  \\
Then $(M,M^{+},(1,1,\cdots,1))$ is a dimension group and $\partial_{e}S(M,M^{+},(1,1,\cdots,1))=\set{\varphi_{k}}^{n}_{k=1}$.  
\end{pro}

\begin{proof}
Since $M$ is dense in $\m{R}^{n}$, $M$ satisfies Riesz property.  \\
Then $(M,M^{+},(1,1,\cdots,1))$ is a dimension group.  \\
We will show that $\partial_{e}S(M,M^{+},(1,1,\cdots,1))=\set{\varphi_{k}}^{n}_{k=1}$.  \\
Let $\varphi$ be in $S(M,M^{+},(1,1,\cdots,1))$ 
and let $g$ be in $M$ arbitrary.  
First, we will show that $\varphi(g_{i})$ converges to $\varphi(g)$ 
for any increasing sequence $\set{g_{i}}^{\infty}_{i=1}$ 
which converges to $g$ and $g_{i}\neq g$ for any $i$.  
By positivity of $\varphi$, 
$\set{\varphi(g_{i})}^{n}_{i=1}$ is an increasing converging sequence.  
Put $\lim_{i\rightarrow \infty}\varphi(g_{i})=\alpha$.  
Conversely, we suppose that $\alpha \neq \varphi(g)$.  
Put $\varepsilon=|\alpha-\varphi(g)|>0$.  
Then $0<\cfrac{1}{n}<\cfrac{\varepsilon}{2}$ for some $n\in \m{N}$.  
Since $g_{i}\rightarrow g$, 
there exists a natural number $M_{1}$ 
such that $0\leq g-g_{i}\leq \cfrac{1}{n}(1,1,\cdots,1)$ for any $i\geq M_{1}$.   
Therefore $0\leq \varphi(g)-\varphi(g_{i})\leq \cfrac{1}{n}$.  
Otherwise, $\lim_{i\rightarrow \infty}\varphi(g_{i})=\alpha$, 
so there exists a natural number $M_{2}$ 
such that $0\leq \alpha-\varphi(g_{i})\leq \cfrac{1}{n}$ for any $i\geq M_{2}$.  
Therefore $|\alpha-\varphi(g)|\leq \cfrac{1}{n}$.  
Hence $0<\varepsilon=|\alpha-\varphi(g)|\leq \cfrac{1}{n}<\cfrac{\varepsilon}{2}$.  
This leads to contradiction.  
As is the same with this previous proof, 
we can show the case of decreasing sequences.  
Second, we will show that $\varphi$ is continuous on $M$ by Euclidean norm of $\m{R}$ 
for any $\varphi$ in $S(M,M^{+},(1,1,\cdots,1))$.  
Let $g$ be in $M$ and let $h_{i}$ be a converging sequence to $g$.  
Since $M$ is dense in $\m{R}^{n}$, we can get some $\set{g_{i}}^{\infty}_{i=1}$ 
which converges to $g$ increasingly and $g_{i}\neq g$ for any $i$.  
Then $\set{2g-g_{i}}^{\infty}_{i=1}$ is a deceasing sequence, converges to $g$, 
and $2g-g_{i}$ is not $g$ for any $i$.  
Let $\varepsilon>0$ arbitrary.  
Then there exists a natural number $M_{3}$ such that $0\leq \varphi(g)-\varphi(g_{i})<\varepsilon$ 
for any $i\geq M_{3}$.  
Since $h_{i}$ converges to $g$, there exists a natural number $M_{4}$ 
such that $g_{M_{3}}\leq h_{i}\leq 2g-g_{M_{3}}$ for any $i\geq M_{4}$.  
Then $\varphi(g_{M_{3}})\leq \varphi(h_{i})\leq \varphi(2g-g_{M_{3}})$.  
Therefore $|\varphi(h_{i})-\varphi(g)|<\varepsilon$ for any $i\geq M_{4}$.  
Since $\varphi(h_{i})$ converges to $\varphi(g)$, $\varphi$ is continuous on $M$.  
Third, we will show that $\set{\varphi(h_{i})}^{\infty}_{i=1}$ is a Cauchy sequence 
if $\set{h_{i}}^{\infty}_{i=1}$ is so.  
Let $\set{g_{i}}^{\infty}_{i=1}$ be a decreasing sequence in $M$ satisfying that $g_{i}\rightarrow (0,0,\cdots,0)$ 
and that $g_{i}\neq (0,0, \cdots,0)$ for any $i$.  
Then $\set{-g_{i}}^{\infty}_{i=1}$ is a increasing sequence, converges to $(0,0,\cdots,0)$, 
and $-g_{i}$ is not $g$ for any $i$.  
Let $\varepsilon>0$ arbitrary.  
Then there exists a natural number $M_{5}$ such that $0\leq \varphi(g_{i})<\varepsilon$ 
for any $i\geq M_{5}$.  
Since $\set{h_{i}}^{n}_{i=1}$, there exists a natural number $M_{6}$ 
such that $-g_{M_{5}}\leq h_{i}-h_{j}\leq g_{M_{5}}$ for any $i,j\geq M_{6}$.  
Then $-\varphi(g_{M_{5}})\leq \varphi(h_{i})-\varphi(h_{j})\leq \varphi(g_{M_{5}})$.  
Therefore $|\varphi(h_{i})-\varphi(h_{j})|<\varepsilon$ for any $i\geq M_{6}$.  
Hence  $\set{\varphi(h_{i})}^{\infty}_{i=1}$ is a Cauchy sequence 
if $\set{h_{i}}^{\infty}_{i=1}$ is so.  
By the second result, the third result and by the density of $M$, 
we can extend $\varphi$ to be a continuous functional $\hat \varphi$ on $\m{R}^{n}$.  
Since $M$ is $\m{Z}$-linear, then $\hat \varphi$ is $\m{R}$-linear.  
We denote by $\hat \varphi_{k}$ the $\m{R}$-linear continuous functional on $\m{R}^{n}$
such that $\hat \varphi_{k}((x_{i})^{n}_{i=1})=x_{k}$.  
Then $\set{\hat \varphi_{k}}^{n}_{k=1}$ is a basis of $\m{R}$-linear functionals on $\m{R}^{n}$.  
Therefore $\hat \varphi=\sum^{n}_{k=1}r_{k}\hat \varphi_{k}$ for some $r_{k}$ in $\m{R}^{n}$.  
Let $\set{e_{k}}^{n}_{k=1}$ be a canonical basis of $\m{R}^{n}$.  
If we consider the decreasing sequence of $M$ which converges to $e_{k}$, 
we can see that $r_{k}\geq 0$.  
Then $\varphi=\sum^{n}_{k=1}r_{k}\varphi_{k}$ for some $r_{k}\geq 0$.  
At last, we will show that $\set{\varphi_{k}}^{n}_{k=1}=\partial_{e}S(M,M^{+},(1,1,\cdots,1))$.  
Let $\varphi_{i}$ and $\varphi\in S(M,M^{+},(1,1,\cdots,1)$.  
We suppose that $0\leq \varphi\leq \varphi_{i}$.  
Then $0\leq \hat \varphi\leq \hat \varphi_{i}$.  
Since $\hat \varphi_{i}$ is in $\partial_{e}S(M,M^{+},(1,1,\cdots,1))$, 
then $\hat \varphi=r \hat \varphi_{i}$ for some $0\leq r\leq 1$.  
Therefore $\varphi=r\varphi_{i}$.  
Hence $\varphi_{i}\in \partial_{e}S(M,M^{+},(1,1,\cdots,1))$.  
Converse is followed by the result that $\set{\varphi_{i}}^{n}_{i=1}$ generates $S(M,M^{+},(1,1,\cdots,1))$.  
Hence $\set{\varphi_{k}}^{n}_{k=1}=\partial_{e}S(M,M^{+},(1,1,\cdots,1))$.  
\end{proof} 

By \ref{pro:dense1} and \ref{pro:dense2}, we can see this proposition.  
\begin{pro}\label{pro:dime}
Let $M$ be an additive subgroup of $\m{R}^{n}$ including $(1,1,\cdots,1)$.  
Put $M^{+}=\set{(g_{i})^{n}_{i=1}\in M:g_{i}>0}\cup\set{(0,0,\cdots,0)}$.  
We suppose that for any $\varepsilon>0$, there exist $\set{a^{\varepsilon}_{i}}^{n}_{i=1}\subset M$ satisfying following properties;\\
(i) $\set{a^{\varepsilon}_{i}}^{n}_{i=1}$ are $\m{R}$-linearly independent. \\ 
(ii)$||a^{\varepsilon}_{i}||< \varepsilon$, 
where $||\cdot||$ is an Euclidean norm of $\m{R}^{n}$.\\
Then $(M,M^{+},(1,1,\cdots,1))$ is a simple dimension group and \\ $\partial_{e}S(M,M^{+},(1,1,\cdots,1))=\set{\varphi_{k}}^{n}_{k=1}$.  
\end{pro}

\begin{proof}
Simplicity is followed by $M^{+}$.  
\end{proof}

We will think about the case $n=2$.  
Let $A$ be an invertible matrix of $M_{2}(\m{R})$.  
We call an additive subgroup $G$ of $\m{R}^{2}$ including $(1,1)$ $A$-$invariant$ if $AG=G$.  \\
Let $G$ be an additive subgroup of $\m{R}^{2}$ including $(1,1)$.  \\
Put $I(G)=\set{A=\left[ 
 \begin{array}{cc}
 a & 0 \\
 0 & b \\
 \end{array} 
 \right],\left[ 
 \begin{array}{cc}
 0 & a \\
 b & 0 \\
 \end{array} 
 \right]\in GL_{2}(\m{R}):AG=G, a>0, b>0}$.    

More general, next proposition follows.  
\begin{pro}\label{pro:rational}
Let $a$ be a positive algebraic number 
and let $b$ be positive rational numbers.  
We suppose that $a\neq b$, $a\neq 1$ and that $b\neq 1$.  
Put $b=\cfrac{p}{q}$, where $gcd(p,q)=1$ and where $p>0,q>0$.  
Put $A=\left[ 
 \begin{array}{cc}
 a & 0 \\
 0 & b \\
 \end{array} 
 \right]$.  
Let $G$ be an $A$-invariant additive subgroup of $\m{R}^{2}$ including $(1,1)$.  
Then $I(G) \supsetneq \set{A^{n}:n\in \m{Z}}$.  
\end{pro}

\begin{proof}
Since $b \neq 1$, $p\neq 1$ or $q\neq 1$.  
We suppose $q \neq 1$.  
Let $\psi(x)=\sum^{n}_{i=0}a_{i}x^{i}$ be a $\m{Z}$-coefficient minimal polynomial of $a$.  
We will show that $\left[ 
 \begin{array}{cc}
 1 & 0 \\
 0 & \cfrac{1}{q^{n}} \\
 \end{array} 
 \right]$ and $\left[ 
 \begin{array}{cc}
 1 & 0 \\
 0 & q^{n} \\
 \end{array} 
 \right]$ are in $I(G)$ for some $n$.  
Let $\alpha$ be an integer 
satisfying that there exist an integer $r$ relatively prime to $q$ 
and a natural number $m$ 
such that $\alpha\sum^{n}_{i=1}a_{i}p^{i}q^{n-i}=rq^{m}$.  
Then there exist a natural number $N>-m+n$ and an integer $s$ 
such that $sr+q^{N}=1$.  
Otherwise $p^{N+m-n}$ and $q^{N+m-n}$ are relatively prime.  
Then $xp^{N+m-n}+yq^{N+m-n}=s$.  
In the same way, $p^{m-n}$ and $q^{m-n}$ are relatively prime, 
so $x^{\prime}p^{m-n}+y^{\prime}q^{m-n}=-s$.
Let $(g,h)$ be an element of $G$.  
Then $(g,h)(\alpha(xA^{N+m-n}+y)\psi(A)+I)=(g,\cfrac{1}{q^{N}}h)$ and 
$(g,h)(\alpha(x^{\prime}A^{m-n}+y^{\prime}I)\psi(A)+I)=(g,\cfrac{1}{q^{N}}h)$.  
 Therefore $\left[ 
 \begin{array}{cc}
 1 & 0 \\
 0 & \cfrac{1}{q^{N}}\\
 \end{array} 
 \right]$ and $\left[ 
 \begin{array}{cc}
 1 & 0 \\
 0 & q^{N} \\
 \end{array} 
 \right]$ are in $I(G)$.
 Hence $I(G) \supsetneq \set{A^{n}:n\in \m{Z}}$. 
\end{proof}

\begin{cor}
Let $a$ be a positive algebraic number 
and let $b$ be positive rational numbers.  
We suppose that $a\neq b$, $a\neq 1$ and that $b\neq 1$.  
Put $b=\cfrac{p}{q}$, where $gcd(p,q)=1$ and where $p>0,q>0$.  
Put $A=\left[ 
 \begin{array}{cc}
 a & 0 \\
 0 & b \\
 \end{array} 
 \right]$.  
There is no $AF$-algebras $\c{A}$ with $2$-dimensional trace space 
such that $\c{A}$ is embedding on $\m{R}^{n}$ and that 
$F(\c{A})=\set{A^{n}:n\in \m{Z}}$.  
\end{cor}

\begin{proof}
Contraversely, we suppose that $F(\c{A})=\set{A^{n}:n\in \m{Z}}$ 
for some $AF$-algebra $\c{A}$ with $2$-dimensional trace space 
such that it satisfies (1).  
Then $K_{0}(\c{A})$ is $A$-invariant.  
By \ref{pro:rational}, $F(\c{A})\supset \set{A^{n}:n\in \m{Z}}$.  
This leads to contradiction.  
\end{proof}

More generally, next proposition follows.  
Let $R$ be an abelian ring.  
We denote by $\m{R}[s,t]$ the $\m{R}$-coefficient polynomial ring generated by $s,t$.  

\begin{pro}\label{pro:rational2}
Let $a,b$ be non-zero positive algebraic numbers.  
We suppose that $a\neq b$.  
Put $A=\left[ 
 \begin{array}{cc}
 a & 0 \\
 0 & b \\
 \end{array} 
 \right]$.  
If there exist $p,q\in \m{Z}$ such that $p\neq 0$, $q\neq 0,1$,and that $p$ and $q$ are
relatively prime and that $\cfrac{p}{q}\in \m{Z}[\beta,\beta^{-1}]$, 
then there is no $AF$-algebras $\c{A}$ with $2$-dimensional trace space 
such that $\c{A}$ is embedding on $\m{R}^{2}$ and that 
$F(\c{A})=\set{A^{n}:n\in \m{Z}}$.  
\end{pro}

\begin{proof}
We denote by $\psi_{1}(t),\psi_{2}(t)$ the $\m{Z}$-coefficient minimal polynomials of $a,b$ respectively.  
Since $a\neq b$, 
$\psi_{1}(t)$ and $\psi_{2}(t)$ are relatively prime in the $\m{Q}$-coefficient polynomial ring $\m{Q}[t]$.  
Then there exist $\m{Z}$-coefficient polynomials $a(t)$, $b(t)$ and $m\in \m{Z}\setminus \set{0}$ such that $a(t)\psi_{1}(t)+b(t)\psi_{2}(t)=m$.   
Therefore $a(b)\psi_{1}(b)=m$.  
We can get $\alpha ,r\in \m{Z}$ satisfying that $\alpha m=rq^{N}$ for some $N\in \m{N}$ 
and that $r$ is relatively prime to $q$.  
Then there exist $s\in \m{Z}$ and $n\in \m{N}$ such that $sr+1=q^{n}$. 
By hypothesis, we can get a $\phi(t)\in P[t,t^{-1}]$ satisfying that $\phi(b)=\cfrac{p}{q}$.  
Since $gcd(p,q)=1$, $gcd(p^{n},q^{n})=1$ and $gcd(p^{n+N},q^{n+N})=1$.  
Then there exist $x,x^{\prime},y,y^{\prime}\in \m{Z}$ 
such that $xp^{n}+yq^{n}=s$ and that $x^{\prime}p^{n+N}+y^{\prime}q^{n+N}=-s$.  
Let $G$ be an $A$-invariant sub-additive group in $\m{R}^{2}$.  
Let $(g,h)$ be an element of $G$.  
Then $(g,h)(q^{n}\alpha(x\phi(A)^{N}+yI)\psi_{1}(A)+I)=(g,q^{n}h)$ and 
$(g,h)(\alpha(x^{\prime}\phi(A)^{n+N}+y^{\prime}I)\psi_{1}(A)+I)=(g,\cfrac{1}{q^{n}}h)$.  
Therefore $\left[ 
 \begin{array}{cc}
 1 & 0 \\
 0 & \cfrac{1}{q^{n}}\\
 \end{array} 
 \right]$ and $\left[ 
 \begin{array}{cc}
 1 & 0 \\
 0 & q^{n} \\
 \end{array} 
 \right]$ are in $I(G)$.
 Hence $I(G) \supsetneq \set{A^{n}:n\in \m{Z}}$. 
 By \ref{pro:embedding}, there is no $AF$-algebras $\c{A}$ with $2$-dimensional trace space 
such that $\c{A}$ is embedding on $\m{R}^{2}$ and that 
$F(\c{A})=\set{A^{n}:n\in \m{Z}}$.  
\end{proof}

\begin{ex}
Let $c_{1}, c_{2}\in \m{Z}$.  
We suppose that $c_{1}>1$ and that $gcd(c_{1},c_{2})=1$.  
If $b>0$ is a solution of the quadratic equation $c_{1}x^{2}+c_{2}x+c_{1}=0$.  
Then $b+\cfrac{1}{b}=-\cfrac{c_{2}}{c_{1}}$.  
Let $a$ be a non-zero positive algebraic number.  
Put $A=\left[ 
 \begin{array}{cc}
 a & 0 \\
 0 & b \\
 \end{array} 
 \right]$.  
By \ref{pro:rational2}, there is no $AF$-algebras $\c{A}$ with $2$-dimensional trace space 
such that $\c{A}$ is embedding on $\m{R}^{2}$ and that 
$F(\c{A})=\set{A^{n}:n\in \m{Z}}$.  
\end{ex}

\begin{lem}\label{lem:monic}
Let $\psi$ be a $\m{Z}$-coefficient monic polynomial 
and let $m$ be in $\m{Z}_{>1}$.  
Then there exist $\varphi_{1}, \varphi_{2}\in \m{Z}[t,t^{-1}]$ and $n(\neq0)\in \m{Z}$ 
such that $m\varphi_{1}(t)+\psi(t)\varphi_{2}(t)+t^{n}=1$.  
\end{lem}

\begin{proof}
Since $\psi$ is monic, $\m{Z}[t]/(m,\psi)$ is finite set.  
By considering $\set{t^{n}}^{\infty}_{n=1}$ in $\m{Z}[t]/(m,\psi)$, 
there exists $n_{1},n_{2}\in \m{N}$ such that $n_{1}\neq n_{2}$ 
and that $t^{n_{1}}=t^{n_{2}}\in \m{Z}[t]/(m,\psi)$.  
Therefore there exist $\varphi^{\prime}_{1}, \varphi^{\prime}_{2}\in \m{Z}[t]$ 
such that $m\varphi^{\prime}_{1}(t)+\psi(t)\varphi^{\prime}_{2}(t)+t^{n_{1}}=t^{n_{2}}$.  
Put $\varphi_{1}(t)=t^{-n_{2}}\varphi^{\prime}_{1}(t)$, $\varphi_{2}(t)=t^{-n_{2}}\varphi^{\prime}_{2}(t)$, 
and $n=n_{1}-n_{2}$.  
Hence $m\varphi_{1}(t)+\psi(t)\varphi_{2}(t)+t^{n}=1$.  
\end{proof}

\begin{pro}\label{pro:rational3}
Let $a,b$ be non-zero positive algebraic numbers.  
We suppose that $a\neq b$, $a\neq 1$, $b\neq 1$, 
minimal $\m{Z}$-coefficient polynomials of $a$ and $b$ are different, 
and that the minimal $\m{Z}$-coefficient polynomial of $b$ is monic.  
Put $A=\left[ 
 \begin{array}{cc}
 a & 0 \\
 0 & b \\
 \end{array} 
 \right]$.  
There is no $AF$-algebras $\c{A}$ with $2$-dimensional trace space 
such that $\c{A}$ is embedding on $\m{R}^{n}$ and that 
$F(\c{A})=\set{A^{n}:n\in \m{Z}}$.  
\end{pro}

\begin{proof}
Let $\psi_{1}$ and $\psi_{2}$ be minimal $\m{Z}$-coefficient polynomials of $a$ and $b$ respectively.  
Since $\psi_{i}$s are minimal and different, 
$\psi_{1}$ and $\psi_{2}$ are relatively prime in $\m{Q}[t]$.  
Then there exist $\phi_{1},\phi_{2}\in \m{Z}[t]$ and $m\in \m{Z}_{>0}$ 
such that $\phi_{1}(t)\psi_{1}(t)+\phi_{2}(t)\psi_{2}(t)=m$.  
If $m=1$, $(A^{n}-1)\phi_{1}(A)\psi_{1}(A)+I
=\left[ 
 \begin{array}{cc}
 1 & 0 \\
 0 & b^{n} \\
 \end{array} 
 \right]$ for any $n\in \m{Z}$.  
We suppose that $m>1$.  
By previous lemma \ref{lem:monic}, 
there exist $\varphi_{1}, \varphi_{2}\in \m{Z}[t,t^{-1}]$ and $n(\neq0)\in \m{Z}$ 
such that $m\varphi_{1}(t)+\psi(t)\varphi_{2}(t)+t^{n}=1$.  
Then $-\varphi_{1}(A)\phi_{1}(A)\psi_{1}(A)+I
=\left[ 
 \begin{array}{cc}
 1 & 0 \\
 0 & b^{n} \\
 \end{array} 
 \right]$
and $A^{-n}\varphi_{1}(A)\phi_{1}(A)\psi_{1}(A)+I
=\left[ 
 \begin{array}{cc}
 1 & 0 \\
 0 & b^{-n} \\
 \end{array} 
 \right]$.  
Let $G$ be an $A$-invariant sub-additive group in $\m{R}^{2}$.  
Let $(g,h)$ be an element of $G$.  
Then $\left[ 
 \begin{array}{cc}
 1 & 0 \\
 0 & b^{n} \\
 \end{array} 
 \right]$ and 
 $\left[ 
 \begin{array}{cc}
 1 & 0 \\
 0 & b^{-n} \\
 \end{array} 
 \right]$ are in $I(G)$ for some $n\in\m{N}$.  
Hence there is no $AF$-algebras $\c{A}$ with $2$-dimensional trace space 
such that $\c{A}$ is embedding on $\m{R}^{n}$ and that 
$F(\c{A})=\set{A^{n}:n\in \m{Z}}$.  
\end{proof}

\begin{pro}\label{pro:transcendal1}
Let $\alpha$ be a non-zero positive transcendal number 
and let $\beta$ be a non-zero positive number.  
We suppose that $\beta\neq 1, \alpha, \cfrac{1}{\alpha}$.   \\
Put $G=\cup_{N\in \m{N}}(\sum^{N}_{i=-N}(\m{Q}(\alpha^{2i},\beta^{2i})+\m{Z}(\alpha^{2i+1},\beta^{2i+1})))$.  \\
$G^{+}=\set{(g,h):g>0,h>0}\cup\set{(0,0)}$, and $u=(1,1)$.  
Then $(G,G^{+},u)$ is a simple dimension group 
satisfying that $\partial_{e}S(G,G^{+},(u))=\set{\varphi_{1},\varphi_{2}}$.  
\end{pro}

\begin{proof}
Let $\varepsilon>0$.  
Then there exist mon-zero rational numbers $q_{1},q_{2}$ 
such that $||(q_{1},q_{1})||<\varepsilon$, $||(q_{2}\alpha^{2},q_{2}\beta^{2})||<\varepsilon$.  
Moreover, $(q_{1},q_{1})$ and $(q_{2}\alpha^{2},q_{2}\beta^{2})$ are $\m{R}$-linearly independent.  
Then $G$ is dense in $\m{R}^{2}$.  
Obviously, $G$ is directed and unperforated, so $G$ is dimension group.  
Simplicity is also.  
\end{proof}

We denote by $\c{A}$ the unital simple $AF$-algebra 
satisfying that \\
 $(K_{0}(\c{A}),K_{0}(\c{A})^{+},[1_{\c{A}}])=(G,G^{+},u)$.

\begin{pro}\label{pro:transcendal2}
Let $\c{A}$ be as above.  
Then $F(\c{A})=\set{\left[ 
 \begin{array}{cc}
\alpha^{2n} & 0  \\
 0 & \beta^{2n} \\
 \end{array} 
 \right]:n\in \m{Z}}$.  
\end{pro}  

\begin{proof}
Since $\left[ 
 \begin{array}{cc}
\alpha^{2n} & 0  \\
 0 & \beta^{2n} \\
 \end{array} 
 \right]G=G$ \\
and $\left[ 
 \begin{array}{cc}
\alpha^{2n} & 0  \\
 0 & \beta^{2n} \\
 \end{array} 
 \right]G^{+}=G^{+}$, $\set{\left[ 
 \begin{array}{cc}
\alpha^{2n} & 0  \\
 0 & \beta^{2n} \\
 \end{array} 
 \right]:n\in \m{Z}}\subset F(\c{A})$.   
Let $\left[ 
 \begin{array}{cc}
 a & 0  \\
 0 & b \\
 \end{array} 
 \right]\in F(\c{A})$.  
Then $a\in \varphi_{1}(G)$ and we can denote $a=\cfrac{1}{\alpha^{k}}\sum^{m}_{i=0}q_{i}\alpha^{i}$ 
for some $q_{i}\in \m{Q}$, $k\in \m{N}\cup\set{0}$.  
Since $a$ is invertible in $\varphi_{1}(G)$, $a=q\alpha^{n}$ for some $q\in \m{Q}$, $n\in \m{Z}$.  
Contraversely, we suppose that $n$ is odd.  
Let $l$ be a rational number satisfying that $lq \not\in \m{Z}$.  
Since $l\in \varphi_{1}(G)$, then $lq\alpha^{n}\in \varphi_{1}(G)$.  
However the set of coefficients of $\alpha^{n}$ is $\m{Z}$ because $n$ is odd.  
It leads to contradiction.  
Therefore $n$ is even.  
Since the set of coefficients of $\alpha$ is $\m{Z}$, then that of $\alpha^{n+1}$ is $q\m{Z}=\m{Z}$.  
Then $q=1$.  
Therefore $a=\alpha^{2n}$ and $b=\beta^{2n}$.  
 We suppose that $\left[ 
 \begin{array}{cc}
 0 & c  \\
 d & 0 \\
 \end{array} 
 \right]\in F(\c{A})$.  
Since $c\varphi_{1}(G)=\varphi_{2}(G)$, there exist $n_{1},n_{2}\in\m{Z}$ and polynomials $p(x),q(x)$ with $\m{Q}$-coefficient such that $cp(\alpha)\alpha^{n_{1}}=1$, $cq(\alpha)\alpha^{n_{2}}=\beta$, and that $p(0)\neq 0$, $q(0)\neq 0$.   
Therefore $c=\alpha^{-n_{1}}\cfrac{1}{p(\alpha)}$ and $\beta=\alpha^{n_{2}-n_{1}}\cfrac{q(\alpha)}{p(\alpha)}=\alpha^{n_{2}-n_{1}}\cfrac{q_{1}(\alpha)}{p_{1}(\alpha)}$, 
where $p_{1}$ and $q_{1}$ are relatively prime and where $p_{1}(0)\neq 0$, $q_{1}(0)\neq 0$.  
We suppose that $deg(p_{1})>0$.  
Let $k\in \m{N}$ satisfying that $deg(p_{1}^{k}(x))>deg(p(x))$.  
Then $\beta^{k}=\alpha^{k(n_{2}-n_{1})}\cfrac{q_{1}^{k}(\alpha)}{p_{1}^{k}(\alpha)}\in c\varphi_{1}(G)$.  
Let $g$ be an element of $c\varphi_{1}(G)$.  
Then we can put $g=\alpha^{n_{3}}\cfrac{r_{1}(\alpha)}{r_{2}(\alpha)}$ 
for some $n_{3}\in \m{N}$ and for some $r_{1}(x),r_{2}(x)\in \m{Q}[x]$, 
where $r_{1}$ and $r_{2}$ are relatively prime and where $r_{2}(0)\neq 0$.  
Since $deg(r_{2}(x))\leq deg(p)$, it leads to contradiction.  
Hence $p_{1}(x)=p_{1}(0)$.  
We consider $\cfrac{1}{\beta}$ and $\cfrac{1}{\beta^{k}}$ for some $k$, in the same with $p_{1}(x)$, 
then $q_{1}(x)=q_{1}(0)$.  
Therefore $\beta=q_{0}\alpha^{n}$ for some $q_{0}\in \m{Q}$, $n\in \m{Z}$.  
Since $\varphi_{2}(G)$ is generated by $q_{0}\alpha^{n}$, $p(\alpha)=p(0)$.  
Moreover, $c\alpha^{n_{0}}\m{Z}\alpha=\cfrac{1}{p(0)}\m{Z}\alpha^{n_{0}-n_{1}+1}$ 
and $c\alpha^{n_{0}}\m{Z}\alpha^{-1}=\cfrac{1}{p(0)}\m{Z}\alpha^{n_{0}-n_{1}-1}$ .  
Since the set of the coefficients of $\alpha^{n_{0}-n_{1}+1}$ is $q^{k_{1}}_{0}\m{Z}$ for some $k_{1}\in \m{Z}$ 
and that of $\alpha^{n_{0}-n_{1}-1}$ is $q_{0}^{k_{2}}\m{Z}$ for some $k_{2}\in \m{Z}$.  
Because $\alpha^{n_{0}-n_{1}+1}\neq \alpha^{n_{0}-n_{1}-1}$, $k_{1}\neq k_{2}$.  
Then we can get the equation $q_{0}^{k_{1}}=\cfrac{1}{p(0)}=q_{0}^{k_{2}}$.  
Therefore $p(0)=q_{0}=1$.   
Then $c=\alpha^{-n_{1}}$ and $\varphi_{2}(G)=\varphi_{1}(G)$ or $\varphi_{2}(G)=\alpha\varphi_{1}(G)$.  
Moreover, If $|n|\geq 2$, then $\alpha^{n+1}$ is not in $\varphi_{2}(G)$.  
Therefore $\beta=1, \alpha, \cfrac{1}{\alpha}$.  
It leads to contradiction in hypothesis.  
Hence $F(\c{A})=\set{\left[ 
 \begin{array}{cc}
\alpha^{2n} & 0  \\
 0 & \beta^{2n} \\
 \end{array} 
 \right]:n\in \m{Z}}$.  
\end{proof}

\begin{pro}\label{pro:transcendal3}
Let $\alpha$ be a non-zero positive transcendal number.   \\
Put $G=\cup_{N\in \m{N}}(\sum^{N}_{i=-N}(\m{Q}(\alpha^{2i},\alpha^{-2i})+\m{Z}(\alpha^{2i+1},\alpha^{-2i-1})+\m{Z}\sqrt[3]{2})(\alpha^{2i+1},0))$.  \\
$G^{+}=\set{(g,h):g>0,h>0}\cup\set{(0,0)}$, and $u=(1,1)$.  
Then $(G,G^{+},u)$ is a simple dimension group 
satisfying that $\partial_{e}S(G,G^{+},(u))=\set{\varphi_{1},\varphi_{2}}$.  
\end{pro}

\begin{proof}
Same with \ref{pro:transcendal1}.   
\end{proof}

We denote by $\c{A}$ the unital simple $AF$-algebra 
satisfying that \\
 $(K_{0}(\c{A}),K_{0}(\c{A})^{+},[1_{\c{A}}])=(G,G^{+},u)$.  

\begin{lem}\label{lem:root2}
Let $\alpha$ be a transcendal number 
and let $p(x),q(x),r(x)$ be in $\m{Q}[x]$.  
If $p(\alpha)(\sqrt[3]{4})+q(\alpha)(\sqrt[3]{2})+r(\alpha)=0$, 
then $p(x)=q(x)=r(x)=0$.  
\end{lem}

\begin{proof}
It is sufficient to show in the case that $p(x),q(x),r(x)$ are $\m{Z}$-coefficient and relatively prime in $\m{Q}(x)$.  
Multiply $(-p(\alpha)r(\alpha)+q^{2}(\alpha))(\sqrt[3]{4})+(-q(\alpha)r(\alpha)+2p^{2}(\alpha))(\sqrt[3]{2})+(r^{2}(\alpha)-2p(\alpha)q(\alpha))$.  
Then we can get $4p^{3}(\alpha)+2q^{3}(\alpha)+r^{3}(\alpha)-6p(\alpha)q(\alpha)r(\alpha)=0$.  
Therefore $r(x)=2r_{1}(x)$ for some $r_{1}(x)\in \m{Z}[x]$.  
Substitute this in $4p^{3}(\alpha)+2q^{3}(\alpha)+r^{3}(\alpha)-6p(\alpha)q(\alpha)r(\alpha)=0$, 
then we can get $2p^{3}(\alpha)+q^{3}(\alpha)+4r^{3}_{1}(\alpha)-6p(\alpha)q(\alpha)r(\alpha)=0$.  
Then  $q(x)=2q_{1}(x)$ for some $q_{1}(x)\in \m{Z}[x]$ and  $p(x)=2p_{1}(x)$ for some $p_{1}(x)\in \m{Z}[x]$.   
Therefore this contradicts to the fact that $p(x),q(x),r(x)$ are relatively prime in $\m{Q}[x]$.  
Hence $p(x)=q(x)=r(x)=0$.  
\end{proof}

\begin{pro}\label{pro:transcendal4}
Let $\c{A}$ be as above.  
Then $F(\c{A})=\set{\left[ 
 \begin{array}{cc}
\alpha^{2n} & 0  \\
 0 & \cfrac{1}{\alpha^{2n}} \\
 \end{array} 
 \right]:n\in \m{Z}}$.  
\end{pro}  

\begin{proof}
Since $\left[ 
 \begin{array}{cc}
\alpha^{2n} & 0  \\
 0 & \cfrac{1}{\alpha^{2n}} \\
 \end{array} 
 \right]G=G$ \\
and $\left[ 
 \begin{array}{cc}
\alpha^{2n} & 0  \\
 0 & \cfrac{1}{\alpha^{2n}} \\
 \end{array} 
 \right]G^{+}=G^{+}$, $\set{\left[ 
 \begin{array}{cc}
\alpha^{2n} & 0  \\
 0 & \cfrac{1}{\alpha^{2n}} \\
 \end{array} 
 \right]:n\in \m{Z}}\subset F(\c{A})$.   
Let $\left[ 
 \begin{array}{cc}
 a & 0  \\
 0 & b \\
 \end{array} 
 \right]\in F(\c{A})$.  
As is the same with \ref{pro:transcendal2}, $b=\alpha^{2n}$ for some $n\in \m{Z}$.  
Then we can put $a=(1+l\sqrt[3]{2})\alpha^{-2n}$ for some $l\in \m{Z}$.  
Contraversely we suppose that $l\neq 0$.  
For any $n\in \m{Z}$, the coefficients of $\alpha^{n}$ includes $\sqrt[3]{2}$ in $a\varphi_{1}(G)$.  
It leads to contradiction to $a\varphi_{1}(G)=\varphi_{1}(G)$.  
Then $l=0$ and $a=\alpha^{-2n}$.  
 We suppose that $\left[ 
 \begin{array}{cc}
 0 & c  \\
 d & 0 \\
 \end{array} 
 \right]\in F(\c{A})$.  
 Then $c\varphi_{1}(G)=\varphi_{2}(G)$ and $\varphi_{1}(G)=d\varphi_{2}(G)$.  
 Therefore $c\in (\varphi_{1}(G)\setminus \set{0})^{-1}\cap \varphi_{2}(G)$ and $d\in \varphi_{1}(G)\cap (\varphi_{2}(G)\setminus \set{0})^{-1}$.  
Since $\varphi_{1}(G)\supset \varphi_{2}(G)$, 
then we will consider that $\varphi_{1}(G)\cap (\varphi_{1}(G)\setminus \set{0})^{-1}$.  
Let $c\in \varphi_{1}(G)\cap (\varphi_{1}(G)\setminus \set{0})^{-1}$.  
Since $c\in \varphi_{1}(G)$, then we can put $c=\cfrac{1}{\alpha^{N}}(\sqrt[3]{2}p(\alpha)+q(\alpha))$ for some $p(x),q(x)\in \m{Q}[x]$ and $N\in \m{Z}_{\geq 0}$.  
Then $\cfrac{1}{c}=\cfrac{\alpha^{N}((\sqrt[3]{4})p^{2}(\alpha)+\sqrt[3]{2}p(\alpha)q(\alpha)+q^{2}(\alpha))}{2p^{3}(\alpha)+q^{3}(\alpha)}$.  
Since $\cfrac{1}{c}\in \varphi_{1}(G)$, then $\cfrac{1}{c}=\cfrac{1}{\alpha^{N_{1}}}(\sqrt[3]{2}p_{1}(\alpha)+q_{1}(\alpha))$ 
for some $p_{1}(x),q_{1}(x)\in \m{Q}[x]$ and $N_{1}\in \m{Z}_{\geq 0}$.  
By \ref{lem:root2}, $p(x)=0$.  
Then $c=\cfrac{q(\alpha)}{\alpha^{N}}$, $\cfrac{1}{c}=\cfrac{\alpha^{N}}{q(\alpha)}$, 
and $q(\alpha)\neq 0$.  
Since $\cfrac{1}{c}\in \varphi_{1}(G)$, then $\cfrac{1}{c}=\cfrac{1}{\alpha^{N_{2}}}(\sqrt[3]{2}p_{2}(\alpha)+q_{2}(\alpha))$ 
for some $p_{2}(x),q_{2}(x)\in \m{Q}[x]$ and $N_{2}\in \m{Z}_{\geq 0}$.  
By \ref{lem:root2}, $p_{2}(x)=0$ and $q_{\alpha}q_{2}(\alpha)-\alpha^{N+N_{2}}=0$.  
Since $\alpha$ is a transcendal number, $q(\alpha)=\alpha^{n_{0}}$ for some $n_{0}\in \m{Z}_{\geq 0}$.  
Then $c=\alpha^{n}$ for some $n \in \m{Z}$.  
But $c\varphi_{1}(G)\neq\varphi_{2}(G)$.  
Then $F(\c{A})\subset \set{\left[ 
 \begin{array}{cc}
\alpha^{2n} & 0  \\
 0 & \cfrac{1}{\alpha^{2n}} \\
 \end{array} 
 \right]:n\in \m{Z}}$. 
 Hence $F(\c{A})=\set{\left[ 
 \begin{array}{cc}
\alpha^{2n} & 0  \\
 0 & \cfrac{1}{\alpha^{2n}} \\
 \end{array} 
 \right]:n\in \m{Z}}$.  
\end{proof}

By replacing $\cfrac{1}{\alpha}$ to $\alpha$ in \ref{pro:transcendal3} and \ref{pro:transcendal4}, 
we can show the next proposition.  

\begin{pro}\label{pro:transcendal5}
Let $\alpha$ be an transcendal number. 
Then there exists an $AF$-algebra such that 
 $F(\c{A})=\set{\left[ 
 \begin{array}{cc}
\alpha^{2n} & 0  \\
 0 & \alpha^{2n} \\
 \end{array} 
 \right]:n\in \m{Z}}$.  
\end{pro}  

\begin{pro}\label{pro:transcendal6}
Let $\alpha$ be a non-zero positive transcendal number .   
Then there exists a unital $AF$-algebra such that \\
$F(\c{A})=\set{\left[ 
 \begin{array}{cc}
\alpha^{n} & 0  \\
 0 & 1 \\
 \end{array} 
 \right]:n\in \m{Z}}$.  
\end{pro}

\begin{proof}
Let $\beta$ be a non-zero positive transcendal number which is not including in the algebraic closure of $\m{Q}[\alpha]$.  
Put $G=\set{(\sum^{N}_{i=-N}a_{i}\alpha^{i},l_{1}+l_{2}\beta):a_{i},l_{1},l_{2}\in \m{Z}}$.  
Then $I(G)=\set{\left[ 
 \begin{array}{cc}
\alpha^{n} & 0  \\
 0 & 1 \\
 \end{array} 
 \right]:n\in \m{Z}}$.  
\end{proof}

\begin{thm}\label{thm:transcendal7}
Let $\alpha$ be a non-zero positive transcendal number and let $\beta >0$.  
Then there exists a unital $AF$-algebra such that \\
$F(\c{A})=\set{\left[ 
 \begin{array}{cc}
\alpha^{n} & 0  \\
 0 & \beta^{n} \\
 \end{array} 
 \right]:n\in \m{Z}}$.  
\end{thm}

\begin{proof}
If $\beta\neq 1$, then replace $\alpha$ to $\alpha^{\frac{1}{2}}$ in \ref{pro:transcendal2}, \ref{pro:transcendal4}, and \ref{pro:transcendal5}.  
If $\beta=1$, by \ref{pro:transcendal6}.  
\end{proof}


\begin{thebibliography}{99}
 
\bibitem{Co}
A. Connes, 
\textit{A factor of type $\r{ II}_{1}$ with countable fundamental group}, 
J. Operator Theory  4, \textbf{1}, (1980), 151--153.  

\bibitem{TK1}
T. Kawahara, 
\textit{Fundamental group of C*-algebras with finite dimensional trace space}, 
J.Operator Theory \textbf{77}(2017) , no.1, 149-170.  

\bibitem{TK2}
\textit{Fundamental group of finite von Neumann algebra with finite dimensional normal trace space}, 
arxiv1608.06611(preprint)

\bibitem{kod1}
K. Kodaka,
\textit{Full projections, equivalence bimodules and automorphisms of stable algebras of 
unital $C^*$-algebras},
J. Operator Theory, \textbf{37} (1997), 357-369. 

\bibitem{kod2}
K. Kodaka,
\textit{Picard groups of irrational rotation $C^*$-algebras},
J. London Math. Soc. (2) \textbf{56} (1997), 179-188. 

\bibitem{kod3}
K. Kodaka,
\textit{Projections inducing automorphisms of stable UHF-algebras}, 
Glasg. Math. J. \textbf{41} (1999), no. 3, 345--354.

\bibitem{MN}
F. Murray and J. von Neumann,
\textit{On rings of operators IV},
Ann. Math. \textbf{44}, (1943), 716--808.

\bibitem{N1}
N. Nawata,
\textit{Fundamental group of simple $C^{*}$-algebras with unique trace $III$}, 
Canad. J. Math. \textbf{64} (2012), 573--587.  

\bibitem{N2}
N. Nawata, 
\textit{Picard groups of certain stably projectionless $C^{*}$-algebras},
J. London Math. Soc. (2) \textbf{88} (2013), 161--180.

\bibitem{NY}
N. Nawata, Y. Watatani,
\textit{Fundamental group of simple $C^{*}$-algebras with unique trace},
Adv. Math. \textbf{225}, (2010), 307--318.  

\bibitem{NY2}
N. Nawata, Y. Watatani,
\textit{Fundamental group of simple $C^{*}$-algebras with unique trace $II$},
Journal of Functional Analysis,  \textbf{260}, (2011), 428--435.  

\bibitem{Po}
S. Popa,   
\textit{string rigidity of $\r{ II}_{1}$ factors arising from malleable actions of $w$-rigid groups}, 
I, Invent. Math , \textbf{165}, (2006), 369--408.  

\bibitem{PoVa}
S. Popa, S. Vaes, 
\textit{Actions of $\m{F}_{\infty}$ whose $\r{ II}_{1}$ factors and orbit equivalence relations have prescribed fundamental group},
J. Amer. Math. Soc. \textbf{23} (2010), 383-403.   

\bibitem{RF}
F. Radulescu, 
\textit{The fundamental group of the von Neumann algebra of a free group with infinitely many generators is $R_{+}\setminus$ 0}, 
J. Amer. Math. Soc.  5, \textbf{3}, (1992), 517--532.  

\bibitem{Vo}
D. Voiculescu, 
\textit{Circular and semicircular systems and free product factors}, 
in:Operator Algebras, Unitary Representations, Enveloping Algebras, and Invariant Theory, in:Progr.Math., \textbf{92}, Birkhauser, Boston, (1990), 45--60.  


\end{thebibliography}
\end{document}